\documentclass[a4paper, 12pt]{amsart}
\usepackage{amsmath,amssymb,enumerate,comment,amsaddr,amsthm,amsfonts,amscd}
\allowdisplaybreaks[2]
\usepackage{bm}
\usepackage{etoolbox}
\makeatletter
\patchcmd{\@maketitle}
  {\ifx\@empty\@dedicatory}
  {\ifx\@empty\@date\else {\vskip3ex \centering\footnotesize\@date\par\vskip1ex}\fi
  \ifx\@empty\@dedicatory}
  {}{}
\patchcmd{\@maketitle}
  {\ifx\@empty\@date\else \@footnotetext{\@setdate}\fi}
  {}{}{}
\makeatother
\usepackage[all]{xy}

\newcommand{\IC}{{\mathbb C}}

\newcommand{\IN}{{\mathbb N}}
\newcommand{\IZ}{{\mathbb Z}}
\newcommand{\IR}{{\mathbb R}}
\newcommand{\IF}{{\mathbb F}}

\newcommand{\cH}{{\mathcal H}}

\newcommand{\cG}{{\mathcal G}}
\newcommand{\fm}{{\mathfrak m}}
\newcommand{\fn}{{\mathfrak n}}

\newcommand{\Ga}{\Gamma}

\newcommand{\Cs}{C$^\ast$}
\newcommand{\rg}{\mathop{{\mathrm C}_{\mathrm r}^\ast}}
\newcommand{\norm}[1]{\left\lVert#1\right\rVert}

\newtheorem{thm}{Theorem}[section]
\newtheorem{prop}[thm]{Proposition}

\newtheorem{lem}[thm]{Lemma}
\newtheorem*{lem*}{Lemma}

\theoremstyle{definition}
\newtheorem{defn}[thm]{Definition}
\newtheorem{exa}[thm]{Example}
\newtheorem{rem}[thm]{Remark}
\title[]{C$^*$-simplicity of relative profinite completions of generalized Baumslag-Solitar groups}
\begin{document}
\author{Miho Mukohara}

\address{the University of Tokyo}

\date{\today}

\begin{abstract}
Suzuki recently gave constructions of non-discrete examples of
locally compact C*-simple groups and Raum showed \Cs-simplicity of the relative profinite completions of the Baumslag-Solitar groups by using Suzuki's results.
We extend this result to some fundamental groups of graphs of groups called generalized Baumslag-Solitar groups.
In this article,
we focus on some sufficient condition to show that these locally compact groups are C$^*$-simple and that KMS-weights of these reduced group C$^*$-algebras are unique.
This condition is an analogue of the Powers averaging property of discrete groups and holds for several currently known constructions of non-discrete C$^*$-simple groups.
\end{abstract}
\maketitle
\section{Introduction}
Group \Cs-algebras and group von Neumann algebras have long been studied.
The group algebra is one of the most fundamental way to construct operator algebras.
The construction of a mutually non-isomorphic family of ${\rm I\hspace{-.01em}I}_1$-factors by McDuff  is an important result in the history.
In addition,
many properties of groups such as amenability or exactness are well-studied in the field of operator algebras.
They are strongly related to group operator algebras and these studies lead to development of important theories including the Baum-Connes theory and classification theories for operator algebras.
It is a natural question whether a given group \Cs-algebra is simple or not.
The reduced group \Cs-algebra is a \Cs-algebra generated by the left regular representation $\lambda$ of a locally compact group.
A locally compact group $G$ is called \Cs-simple,
if its reduced group \Cs-algebra $\rg(G)$ is a simple \Cs-algebra.
\Cs-simplicity of discrete groups has been studied since Powers proved that the non-commutative free group $\IF_2$ is \Cs-simple \cite{Po},
and as of today, satisfactory characterizations related to boundary actions and the unique trace property of reduced group \Cs-algebras have been found for discrete groups \cite{K-K}, \cite{Haa}.

On the other hand,
\Cs-simplicity of non-discrete locally compact groups is not well understood compared to that of discrete groups. 
The existence of a non-discrete example of \Cs-simple groups was asked by de la Harpe in \cite{Har1},
and Suzuki constructed the first example of it in \cite{S1}.
As of today,
some examples were found in \cite{S1}, \cite{S2}, and \cite{R2}.
By using the construction in \cite{S1}, Raum \cite{R2} showed the relative profinite completion of the Baumslag-Solitar group BS$(n,m)$ is \Cs-simple,
if $|n|,|m|\geq 2$.
Our main result is a generalization of that.
The Baumslag-Solitar group BS$(n,m)$ is generated by two elements $a,t$ with a relation 
$ta^{m}t^{-1}=a^{n}$,
and it has a natural  action on a tree whose vertex stabilizers and edge stabilizers are isomorphic to $\IZ$.
Generally,
groups with such actions on trees are called generalized Baumslag-Solitar groups.
In this article,
we show \Cs-simplicity of the closures of generalized Baumslag-Solitar groups in the topology of the automorphism groups on trees.
Because every locally compact \Cs-simple group is totally disconnected \cite{R1},
it is natural to consider C$^*$-simplicity for closed subgroups of automorphism groups on trees.

For discrete groups,
Haagerup \cite{Haa} and Kennedy \cite{Ke} proved that a group $\Ga$ is \Cs-simple,
if and only if it has the Powers averaging property,
that is,
for any element $a$ in the reduced group \Cs-algebra $\rg(\Ga)$ with the canonical trace $\tau$,
$\tau(a)$ is approximated by elements in the convex hull of the set
$\{\lambda_{s}x{\lambda_s}^* \in \rg(\Ga) \mid s\in \Ga\}$.
In the proof of our main theorem,
we focus on the canonical conditional expectation $E_K$ from the reduced group \Cs-algebra $\rg(G)$ of totally disconnected group $G$ onto the \Cs-subalgebra $\rg(K)$ and the averaging projection $p_K$ corresponding to a compact open subgroup $K$ of $G$.
The following condition can be confirmed in a direct way.

\begin{lem*}[Lemma \ref{lem:ce1}]
Let $(K_\nu)_{\nu\in N}$ be a decreasing net of compact open subsets of $G$ and $(A_\nu)_{\nu\in N}$ be an increasing net $(A_\nu)_{\nu\in N}$ of \Cs-subalgebras of $\rg(G)$ with $\bigcap_{\nu}K_{\nu}=\{e\}$ and $\overline{\bigcup_{\nu}A_\nu}=\rg(G)$.
If every conditional expectation $E_{\nu}:=E_{K_\nu}$ and averaging projection $p_{\nu}:=p_{K_\nu}$
satisfy the following conditions,
then $G$ is \Cs-simple.
 \begin{itemize}
\item[(1)]$ p_{\nu}\in A_\nu$.  
\item[(2)]  
          For any $\epsilon>0$ and self adjoint element $x\in A_\nu$,
          there are $g_{1}, g_{2},\dots, g_{n}\in G$ such that
          \[
             \norm{\frac{1}{n}\sum_{i=1}^{n}\lambda_{g_i}(x-E_{\nu}(x))\lambda_{g_i}^*}<\epsilon,
             \quad 
             \norm{p_{\nu}-\frac{1}{n}\sum_{i=1}^{n}\lambda_{g_i}p_{\nu}\lambda_{g_i}^{*}}<\epsilon.
          \]   
 \end{itemize}
Moreover,
if $C_{cc}(G)\subset{\bigcup_{\nu}A_\nu}$,
and each $g_{1}, g_{2},\dots, g_{n}$ can be taken in the kernel of the modular function on $G$,
then the Plancherel weight is a unique $\sigma^{\varphi}$-KMS-weight on $\rg(G)$ up to scalar multiple.
\end{lem*}

The assumption of the above lemma is an analogue of the Powers averaging property with respect to conditional expectations,
and not only groups in our main theorem,
but also two kinds of non-discrete \Cs-simple groups constructed by Suzuki \cite{S1}, \cite{S2} satisfy it.
\section{Preliminaries}
\subsection{Weights on \Cs-algebras}
Let $A$ be a \Cs-algebra and $A^+$ be the set of positive elements in $A$.
A map $\psi\colon A^{+}\rightarrow [0,\infty]$ is called a weight on $A$,
if for any $x,y\in A^{+}$ and $r>0$,
we have
\[
\psi(x+y)=\psi(x)+\psi(y),
\quad
\psi(rx)=r\psi(x).
\]
For a weight $\psi$ on $A$,
we set 
\begin{align*}
\fm_{\psi}^{+}&:=\{x\in A^{+}\mid \psi(x)<\infty\},\\
\fn_{\psi}&:=\{y\in A\mid y^{*}y\in\fm_{\psi}^{+}\},\\
\fm_{\psi}&:=\fn_{\psi}^{*}\fn_{\psi}={\mathrm{span}}_{\IC}\fm_{\psi}^{+},
\end{align*}
as in Definition 2.17 of \cite{R1}.
This $\fn_{\psi}$ is a left ideal of $A$ and this $\fm_{\psi}$ is a subalgebra of $A$.
There is a linear functional from $\fm_{\psi}$ to $\IC$, which is an extension of $\psi|_{\fm_{\psi}^{+}}$.
We say that $\psi$ is densely defined,
if $\fm_{\psi}$ is dense in $A$.
In this article,
we suppose that every weight $\psi$ is non-zero, lower semi-continuous and densely defined.

\subsubsection{KMS-weights}
For a continuous one-parameter group $(\sigma_t)_{t\in \IR}$ of $\ast$-automorphisms of $A$,
the set of analytic elements is dense in $A$. (See Section 1 of \cite{Ku}.)
A weight  $\psi$ is called a $\sigma$-KMS-weight,
if 
\[
\psi\circ\sigma_{t}=\psi,
\quad
\psi(x^{*}x)=\psi(\sigma_{\frac{i}{2}}(x)\sigma_{\frac{i}{2}}(x)^{*})
\]
for all $t\in \IR$ and analytic elements $x\in A$. 
(See Section 2.6.3 of {\cite{R1}}.)

The following proposition holds.
\begin{prop}[{\cite[Proposition 2.24]{R1}}]
Let $(\sigma_t)_{t\in \IR}$ be a one-parameter group of $\ast$-automorphisms on a \Cs-algebra $A$,
and $\psi$ be a $\sigma$-KMS-weight.
Any analytic elements $x, y\in\fn_{\psi}\cap\fn_{\psi}^{*}$ satisfy $\psi(xy)=\psi(y\sigma_{-i}(x))$.
\end{prop}

\subsection{Reduced group \Cs-algebras}
In this this article,
let $G$ be a locally compact group with a left Haar measure $\mu$ and $\lambda$ be a left regular representation on a Hilbert space $L^{2}(G):=L^{2}(G,\mu)$.
The representation $\lambda$ extends to a $\ast$-representation of $C_{c}(G)$ on $B(L^{2}(G))$ as follows.
For every $\xi\in L^{2}(G),
f\in C_{c}(G)$,
and $g\in G$,
we have
\[
   \lambda (f)\xi (g):=\int_{G} f(h) \xi(h^{-1}g) d\mu (h).
\]
The reduced group \Cs-algebra $\rg(G)$ of G is the norm closure of $\lambda (C_c(G))$.

\subsubsection{Averaging projections}
If $K$ is a compact open subgroup of $G$,
a projection 
\[
p_{K}:=\lambda(\frac{1}{\mu(K)}\chi_{K})
\]
induced by an indicator function $\chi_{K}\in C_{c}(G)$ of $K$ is called the averaging projection.
(See Section 2.6.2 of \cite{R1} and Section 2 of \cite{S2}.)
These averaging projections satisfy $p_{K}\geq p_{L}$ for any compact open subgroups $K, L$ with
$K\subset{L}$,
since $p_K$ is the orthogonal projection onto the subspace of $\lambda(K)$-fixed points in $L^{2}(G)$.
When $G$ is totally disconnected,
a family $\Omega$ of compact open subgroups of $G$ generate a neighborhood basis of $e\in G$,
and $\{p_{K}\}_{K\in\Omega}$ give approximate units of $\rg(G)$.

\subsubsection{Conditional expectations}
For an open subgroup $H$ of $G$,
we may identify $\rg(H)$ with the \Cs-subalgebra of $\rg(G)$ generated by $\lambda(C_{c}(H))$.
It is well-known that the restriction map $E_{K}\colon C_{c}(G)\rightarrow C_{c}(H)$ extends to a faithful conditional expectation from $\rg(G)$ onto $\rg(H)$.
(See Section 2.5 of \cite{B-O}.)
This conditional expectation is also denoted by $E_K$ in this article.

\subsubsection{The Plancherel weight}
Let $\Delta$ be the modular function of $G$.
The modular flow  $(\sigma_t^{\varphi})_{t\in \IR}$ on $\rg(G)$ is defined as
\[
  \sigma_t^{\varphi}(f)(g):=\Delta(g)^{it}f(g),
\]
for any $f\in C_{c}(G)$ and $g\in G$.
The map $\varphi\colon C_{c}(G)\ni f \mapsto f(e)\in\IC$ extends to a $\sigma^\varphi$-KMS-weight on $\rg(G)$.
(See Section 2.6.2 of \cite{R1}.)
This is a restriction of the Plancherel weight on the group von Neumann algebra $L(G)$ 
(see \cite{Ta} for a definition),
which is also called the Plancherel weight.
When $G$ is totally disconnected,
for a set $\Omega$ of all compact open subgroups in $G$,
we set a subalgebra
\[
C_{cc}(G):=\bigcup_{K\in\Omega}{p_{K}C_{c}(G)p_{K}}
\]
of $\rg(G)$ as in Section 4 of \cite{S2}.
The following propositions about general $\sigma^\varphi$-KMS-weights are known.
We assume that $G$ is a totally disconnected group and $(\sigma_t^{\varphi})_{t\in \IR}$ is a modular flow on $\rg(G)$ in the following propositions.
\begin{prop}[{\cite[Section 4]{S2}}]\label{prop;domain}
If $\psi$ is a $\sigma_t^{\varphi}$-KMS-weight on $\rg(G)$,
then $C_{cc}(G)\subset{\fm_{\psi}}$.
\end{prop}

\begin{prop}[{\cite[Lemma 2.23]{R1}}, {\cite[proof of Theorem 4.1]{S2}}]\label{prop;unique}
If a $\sigma_t^{\varphi}$-KMS weight $\psi$ on $\rg(G)$ satisfies $\psi(\lambda_{g}p_{K})=0$
for all $g\in G\setminus K$ and compact subgroups $K$,
then $\frac{1}{\psi(p_{K})\mu(K)}\psi$ is the Plancherel weight on $\rg(G)$ for any compact open subgroup $K$.
\end{prop}

\subsection{Constructions of non-discrete \Cs-simple groups}
In this subsection, we explain two ways to construct non-discrete example of \Cs-simple groups,
which are found by Suzuki in \cite{S1}, \cite{S2}.

First examples are constructed by the following proposition.

\begin{prop}[{\cite[Proposition]{S1}}]\label{prop;elementary}
Let $G$ be a locally compact group.
Assume we have a decreasing sequence $(K_n)_{n=1}^\infty$ of compact open subgroups of $G$
and an increasing sequence $(L_n)_{n=1}^\infty$ of clopen subgroups of $G$ with the following properties.
\begin{itemize}
\item Each $L_n$ contains $K_n$ and normalizes it.
\item The quotient groups $L_n/K_n$ are \Cs -simple.
\item The intersection $\bigcap_{n=1}^\infty K_n$ is the trivial subgroup $\{e\}$.
\item The union $\bigcup_{n=1}^\infty L_n$ is equal to $G$.
\end{itemize}
Then $G$ is \Cs-simple and has the unique trace property.
\end{prop}
The proposition above shows \Cs-simplicity of locally compact groups like $(\bigoplus_{n=1}^\infty \Gamma_n)\rtimes \prod_{n=1}^\infty F_n$.
(See Theorem in \cite{S1}.)
Here, 
$F_n$ are isomorphic to $\IZ_2$ and discrete groups $\Ga_n$ with $\IZ_2$ actions are induced by the splitting  short exact sequence
\[
\xymatrix@=10pt{ 0 \ar[r] &\Ga_n\ar[r]&\IZ\ast\IZ_2\ar@<0.5ex>[r]&\IZ_2\ar[r]\ar@<0.5ex>[l]&0\\}.
\]

The second construction is established in Section 3 of \cite{S2}.
For a totally disconnected group $G$ and the set $\Omega$ of all compact open subgroups of $G$,
let $\Upsilon_n$; $n\in \mathbb{N}$, be pairwise distinct copies of the group
\[\bigoplus_{K\in\Omega}\bigoplus_{G/K} \IZ_2,\]
equipped with $G$ action induced by the left translation action on $G/K$.
Similarly,
let $\Xi_n$; $n\in \mathbb{N}$, be pairwise distinct copies of $\IZ$ with the trivial $G$-action.
Set
 \[\Gamma_1:= \Upsilon_1,\quad
\Lambda_1:= \Gamma_1 \ast \Xi_1,\]
\[\Gamma_{n+1}:= \Lambda_n \times \Upsilon_{n+1},\quad
\Lambda_{n+1}:=\Gamma_{n+1} \ast \Xi_{n+1},\]
for all $n\in\IN$.
Define $\Lambda$ to be the inductive limit of the sequence,
\[\Gamma_1< \Lambda _1 < \Gamma_2 < \Lambda_2 < \cdots\]
of discrete groups with canonical $G$ actions,
and set
\[\cG:= \Lambda \rtimes G.\]
The following theorems hold.
\begin{thm}[{\cite[Theorem 3.1]{S2}}]\label{prop;non-elementary}
The locally compact group $\cG$ is \Cs-simple.
\end{thm}

\begin{thm}[{\cite[Theorem 4.1]{S2}}]
Up to scalar multiple, the Plancherel weight $\varphi$ is the only
$\sigma^\varphi$-KMS weight on $\rg(\cG)$.
When $\cG$ is non-unimodular, there is no tracial weight on $\rg(\cG)$.
\end{thm}

\subsection{Graphs of groups}
In this subsection,
we introduce the basic notations about graphs of groups and their fundamental groups.
We use the same notations and definitions as in \cite{Se} and \cite{H-P}. 

Let $X$ be a connected graph with a vertex set $V(X)$ and an edge set $E(X)\subset{V(X)\times V(X)}$.
For an edge $x\in E(X)$,
vertices $o(x)$ and $t(x)$ denote the origin and the terminus of $x$.
We write $y=\overline{x}$,
When edges $x$ and $y$ satisfy $o(x)=t(y)$ and $o(y)=t(x)$.
In this article,
every edge $x\in E(X)$ has $\overline{x}\in E(X)$ for every graph $X$.
An orientation $A$ of $X$ is a subset of $E(X)$ containing exactly one of $x, \overline{x}$.
When $X$ is a graph with an orientation $A$,
we define the function $e\colon E(X)\rightarrow\{0, 1\}$ as
\[
e(x):=
\begin{cases}
0&x\in A\\
1&\overline{x}\in A.
\end{cases}
\]

\begin{defn}[{\cite[Section 4]{H-P}}]
A graph of groups $(G, Y)$ consists of the following.
\begin{itemize}
\item A non-empty connected graph $Y$.
\item Two families of groups $(G_{P})_{P\in V(Y)}$ and $(G_{y})_{y\in E(Y)}$ with $G_{y}=G_{\overline{y}}$ for all $y\in E(Y)$.
\item A family of monomorphisms $\{\iota_{y}\colon G_{y}\rightarrow G_{t(y)}\}_{y\in E(Y)}$.
\end{itemize}
\end{defn}

\subsubsection{The groups $F(G, Y)$}
We define the group $F(G, Y)$ for a graph of groups $(G, Y)$ as follows.
(See Section 5.1 of \cite{Se}.)
Let $\Gamma$ be the free product of the vertex groups $(G_{P})_{P\in V(Y)}$ and the free group with basis $E(Y)$.
The group $F(G, Y)$ is the quotient of $\Gamma$ by the normal subgroup generated by
\[ y\overline{y}\quad
   \text{and}
   \quad y\iota_{y}(a)y^{-1}(\iota_{\overline{y}}(a))^{-1}
\]
for all $y\in E(Y)$ and $a\in G_y$.

\subsubsection{Words of type $c$}
Let $c=\{y_{1}, y_{2},\dots, y_{n}\}$ be a path of $Y$,
where $y_{1}, y_{2},\dots, y_{n}$ are edges of $Y$ with $t(y_{i})=o(y_{i+1})$.
For a sequence $\mu=(r_{0}, r_{1},\dots, r_{n})$ of elements $r_{0}\in G_{o(y_{1})}$ and $r_{i}\in G_{t(y_{i})}$,
the element 
\[r_{0}y_{1}r_{1}y_{2}\cdots y_{n}r_{n}\]
of $F(G, Y)$ is said to be associated with the word $(c, \mu)$ and denoted by $|c, \mu|$.
(See Definition 9 of \cite{Se}.)

A pair $(c, \mu)$ is called reduced if either of the following holds.
\begin{itemize}
\item The length $n$ of $c$ is not equal to $0$ and $r_{i-1}\notin\iota_{\overline{y_{i}}}(G_{y_{i}})$ for every index $i$ with $y_{i}=\overline{y_{i-1}}$.
\item The length $n=0$ and $r_{0}\neq0$.
\end{itemize}
When $(c, \mu)$ is reduced,
we have $|c,\mu|\neq e$.
(See Theorem 11 of \cite{Se}.)

\subsubsection{The fundamental groups $\pi_{1}(G, Y)$}
Let $P_{0}$ be an element of $V(Y)$.
The group $\pi_{1}(G, Y, P_{0})$ is a subgroup of $F(G, Y)$ generated by the set of elements $|c, \mu|$ of closed paths $c$ from $P_{0}$ to $P_{0}$.
This is called the fundamental group of  $(G, Y)$ at $P_{0}$.
(See section 5.1 of \cite{Se}.)

Let $T$ be a maximal subtree of $Y$.
The fundamental group $\pi_{1}(G, Y, T)$ of $(G, Y)$ at $T$ is the quotient of $F(G, Y)$ by the normal subgroup generated by the set of edges of $T$.
In $\pi_{1}(G, Y, T)$,
the element induced by $y\in E(Y)$ is denoted by $g_y$.
(See section 5.1 of \cite{Se}.)

The following proposition holds.

\begin{prop}[{\cite[Proposition 20]{Se}}]\label{prop;fg}
For any maximal subtree $T$ of $Y$ and $P_{0}\in V(Y)$,
the canonical inclusion $i_{P_0}\colon \pi_{1}(G, Y, P_{0})\rightarrow F(G, Y)$ and the canonical quotient map
$q_{T}\colon F(G, Y)\rightarrow \pi_{1}(G, Y, T)$ induce an isomorphism of $\pi_{1}(G, Y, P_{0})$ onto $\pi_{1}(G, Y, T)$.
\end{prop}
This proposition shows that the isomorphism classes of $\pi_{1}(G, Y, P_{0})$ and $\pi_{1}(G, Y, T)$ are independent of the choice of $P_{0}$ and $T$.
Thus we write $\pi_{1}(G, Y)$ instead of $\pi_{1}(G, Y, P_{0})$ or $\pi_{1}(G, Y, T)$ if no confusion arises.
When $c$ is a closed path of $Y$,
we have $q_{T}(|c, \mu|)\neq e$ for a reduced form $(c, \mu)$.
(See Corollary 3 of \cite{Se}.)

\begin{exa}[Amalgamated free products]
If a connected graph $Y$ consists of two vertices $\{P, Q\}$ and two edges $\{y, \overline{y}\}$,
then the fundamental group $\pi_{1}(G, Y)$ is isomorphic to the amalgamated free product $G_{P}\ast_{G_y}G_{Q}$ with respect to the inclusions $\iota_{y}$ and $\iota_{\overline{y}}$.
(See Section 1.2 \cite{Se} for the definition.)
In this case,
every element $g\in G_{P}\ast_{G_y}G_{Q}\setminus{\{e\}}$ has a form
\[
g=a_{1}a_{2}\cdots a_{n}c 
\quad\text{with}
\quad a_{i}\in (G_{P}\setminus{\iota_{y}(G_{y})})\cup(G_{Q}\setminus{\iota_{\overline{y}}(G_{y})}),
\quad c\in G_{y}.
\]
In addition,
the sequence $(a_{1}, a_{2},\dots, a_{n}, c)$ can be chosen as follows.
\begin{itemize}
\item When $n=0$, this $c$ is not equal to $e$.
\item
We have $a_{i+1}\in G_{P}\setminus{\iota_{y}(G_{y})}$ if and only if $a_{i}\in G_{Q}\setminus{\iota_{\overline{y}}(G_{y})}$ for any $i$.
\end{itemize}
Generally,
a sequence $(a_{1}, a_{2},\dots, a_{n}, c)$ of words of an amalgamated free product is said to be reduced,
if the above conditions hold.
Since every sequence of reduced words of $G_{P}\ast_{G_y}G_{Q}$ corresponds to a reduced form $(c, \mu)$ of $(Y, G)$,
an element $g\in G_{P}\ast_{G_y}G_{Q}$ induced by reduced words is not equal to $e$.
\end{exa}

\begin{exa}[HNN extensions]
If a graph $Y$ consists of one vertex $P$ and two edges of loops $\{y, \overline{y}\}$,
then the fundamental group $\pi_{1}(G, Y)$ is isomorphic to the HNN extension $G_{P}\ast_{\theta}$ with respect to the canonical isomorphism $\theta$ from $\iota_{y}(G_{y})$ to $\iota_{\overline{y}}(G_{y})$.
(See Section 1.4 of \cite{Se} for the definition.)
Every element $g\in G_{P}\ast_{\theta}\setminus{\{e\}}$ has the form
\[
g=a_{0}t^{\epsilon_1}a_{1}t^{\epsilon_2}a_{2}\cdots t^{\epsilon_n}a_{n}
\]
with $a_{i}\in G_{P}$, $\epsilon_{i}\in\{1,-1\}$, and the stable letter $t$ of the HNN extension.
We can take these $a_i$ and $\epsilon_{i}$ as follows.
\begin{itemize}
\item If $n=0$, then $a_{0}\neq e$.
\item If an index $i$ satisfies $(\epsilon_{i-1}, \epsilon_{i})=(1, -1)$, then $a_{i-1}\notin\iota_{y}(G_{y})$.
\item If an index $i$ satisfies $(\epsilon_{i-1}, \epsilon_{i})=(-1, 1)$, then $a_{i-1}\notin\iota_{\overline{y}}(G_{y})$.
\end{itemize}
Generally,
a sequence $(a_{0}, t^{\epsilon_1}, a_{1}, t^{\epsilon_2}, a_{2},\dots,  t^{\epsilon_n}, a_{n})$ of an HNN-extension is called reduced words if the above conditions hold.
For the same reasons as for amalgamated free products, if $g\in G_{P}\ast_{\theta}$ is induced by reduced words,
then $g\neq e$.
\end{exa}

\subsubsection{The universal covering}
Fix a maximal subtree $T$ of $Y$ and an orientation $A\subset{E(Y)}$ of $Y$.
For every $P\in V(Y)$, $G_{P}$ is a subgroup of $\pi_{1}(G, Y, T)$ in a natural way.
The universal covering $\tilde{X}:=\tilde{X}(G, Y, T)$ of $(G, Y, T)$ is a graph with a vertex set 
$V(\tilde{X}):=\bigsqcup_{p\in V(Y)}\pi_{1}(G, Y, T)/G_{P}$ and an edge set 
$E(\tilde{X}):=\bigsqcup_{y\in E(Y)}\pi_{1}(G, Y, T)/\iota_{y}(G_{y})$,
where
\[o(g\iota_{y}(G_{y}))=gg_{y}^{e(y)}G_{o(y)},\quad
   t(g\iota_{y}(G_{y}))=gg_{y}^{1-e(y)}G_{t(y)}
\]
for any $g\in \pi_{1}(G, Y, T)$ and $y\in E(Y)$.
This $\tilde{X}$ has an action of $\pi_{1}(G, Y, T)$ induced by the left multiplication.
It is known that the universal covering $\tilde{X}$ is a tree.
(See Theorem 12 of \cite{Se}.)

\section{C$^*$-simplicity and the Powers averaging property for conditional expectations}
In this section,
we suppose that $G$ is a totally disconnected locally compact group and $\sigma^{\varphi}$ is the modular flow.
First, we give a sufficient condition for \Cs-simplicity of $G$.

\begin{lem}\label{lem:ce1}
Let $(K_\nu)_{\nu\in N}$ be a decreasing net of compact open subsets of $G$ and $(A_\nu)_{\nu\in N}$ be an increasing net $(A_\nu)_{\nu\in N}$ of \Cs-subalgebras of $\rg(G)$ with $\bigcap_{\nu}K_{\nu}=\{e\}$ and $\overline{\bigcup_{\nu}A_\nu}=\rg(G)$.
If every conditional expectation $E_{\nu}:=E_{K_\nu}$ and averaging projection $p_{\nu}:=p_{K_\nu}$
satisfy the following conditions,
then $G$ is \Cs-simple.
 \begin{itemize}
\item[(1)] The averaging projection $p_{\nu}$ is in $A_\nu$.  
\item[(2)]  
          For any $\epsilon>0$ and self adjoint element $x\in A_\nu$,
          there are $g_{1}, g_{2},\dots, g_{n}\in G$ such that
          \[
             \norm{\frac{1}{n}\sum_{i=1}^{n}\lambda_{g_i}(x-E_{\nu}(x))\lambda_{g_i}^*}<\epsilon,
             \quad 
             \norm{p_{\nu}-\frac{1}{n}\sum_{i=1}^{n}\lambda_{g_i}p_{\nu}\lambda_{g_i}^{*}}<\epsilon.
          \]   
 \end{itemize}
Moreover,
if $C_{cc}(G)\subset{\bigcup_{\nu}A_\nu}$,
and each $g_{1}, g_{2},\dots, g_{n}$ can be taken in the kernel of the modular function on $G$,
then the Plancherel weight is a unique $\sigma^{\varphi}$-KMS-weight on $\rg(G)$ up to scalar multiple.
\end{lem}

\begin{proof}
Let $I$ be a non-zero norm closed two-sided ideal of $\rg(G)$.
To prove \Cs-simplicity of $G$,
it suffices to show that the averaging projection $p_{\nu}$ is in $I$ for sufficiently large $\nu$,
since the net $(p_{\nu})_\nu$ gives approximate units of $\rg(G)$.

Since $\bigcup_{\nu}A_{\nu}\cap I\neq0$,
we can take a positive element $x\in\bigcup_{\nu}A_{\nu}\cap I$. There is $\nu_{0}\in N$ such that 
$p_{\nu}xp_{\nu}\neq0$ and $x\in A_{\nu}\cap I$ for any $\nu\geq\nu_0$.
Fix $\nu\geq\nu_0$ and we show $p_{\nu}\in I$.
We may assume $\norm{E_{\nu}(x)}=1$ and $x=p_{\nu}xp_{\nu}$ by replacing positive $x\in A_{\nu}\cap I$,
since $E_{\nu}$ is faithful.

Every element $y\in p_{\nu}C_{c}(G)p_{\nu}$ has the form $y=\sum_{s\in F}\alpha_{s}p_{\nu}\lambda_{s}p_{\nu}$ and $E_{\nu}(y)=\alpha_{e}$,
where $F$ is a finite subset of $G$ containing $e$ whose elements have mutually distinct $K_\nu$-double cosets and  $\alpha_{s}\in\IC$. (See Proof of Theorem 3.1 of \cite{S2}.)
Since $x$ is approximated by positive elements $y\in p_{\nu}C_{c}(G)p_{\nu}$ with $\norm{E_{\nu}(y)}=1$,
we have $E_{\nu}(x)=p_\nu$.

By using the assumption of the lemma, for $\epsilon>0$ we chose $g_{1}, g_{2},\dots, g_{n}\in G$ satisfying
\[
       \norm{\frac{1}{n}\sum_{i=1}^{n}\lambda_{g_i}(x-E_{\nu}(x))\lambda_{g_i}^*}<\epsilon,
       \quad 
       \norm{p_{K}-\frac{1}{n}\sum_{i=1}^{n}\lambda_{g_i}p_{\nu}\lambda_{g_i}^{*}}<\epsilon.
\]   
Thus
\begin{align*}
\quad\norm{p_{\nu}-\frac{1}{n}\sum_{i=1}^{n}\lambda_{g_i}x\lambda_{g_i}^*}
<2\epsilon.
\end{align*}
When $\epsilon<\frac{1}{2}$,
the element $\sum_{i=1}^{n}p_{\nu}\lambda_{g_i}x\lambda_{g_i}^{*}p_{\nu}$ of $I$ is invertible in $p_{\nu}\rg(G)p_{\nu}$.
We get $p_{\nu}\in I$.

Next,
we show the uniqueness of a $\sigma^{\varphi}$-KMS-weight.
Let $\psi$ be a $\sigma^{\varphi}$-KMS-weight,
$L$ be a compact open subgroup of $G$,
and $s\in G\setminus L$.
By Proposition \ref{prop;unique}, it suffices to show that $\psi(\lambda_{s}P_{L})=0$.
We prove this equation in the same way as in the proof of Theorem 4.1 in \cite{S2}.
Set $x_1:=p_{L}\lambda_{s}p_{L}+p_{L}\lambda_{s^{-1}}p_{L}$.
We can take $\nu\in N$ with $K_\nu\subset{L}$ and $x_{1}\in A_\nu$.
For $\epsilon>0$,
we have $g_{1}, g_{2},\dots, g_{n}$ in the kernel of the modular function with
 \[
      \norm{\frac{1}{n}\sum_{j=1}^{n}\lambda_{g_j}x_{1}\lambda_{g_j}^*}<\epsilon,
      \quad 
      \norm{p_{K}-\frac{1}{n}\sum_{j=1}^{n}\lambda_{g_j}p_{\nu}\lambda_{g_j}^{*}}<\epsilon,
 \]   
since $E_{K}(x_1)=0$.

Set
\[
 y:=\frac{1}{n}\sum_{j=1}^{n}p_{\nu}\lambda_{g_j}^{*}p_{\nu}\lambda_{g_j}p_{\nu}\in p_{\nu}C_{c}(G)p_{\nu},
\]
then 
\[
\left|\psi\left(p_{\nu}-y\right)\right|
=\left|\psi\left(p_{\nu}-\frac{1}{n}\sum_{j=1}^{n}p_{\nu}\lambda_{g_j}p_{\nu}\lambda_{g_j}^{*}p_{\nu}\right)\right|
<\psi(p_{\nu})\epsilon.
\]
The equation above follows from the KMS-condition of $\psi$,
since $\sigma^{\varphi}_{-i}(p_{\nu}\lambda_{g_j}p_{\nu})=p_{\nu}\lambda_{g_j}p_{\nu}$.
Similarly,
\[
\left|\psi(yx_{1})\right|
=\left|\psi\left(\frac{1}{n}\sum_{j=1}^{n}p_{\nu}\lambda_{g_j}x_{1}\lambda_{g_j}^{*}p_{\nu}\right)\right|
<\psi(p_{\nu})\epsilon.
\]

Therefore, we get
\begin{align*}
\left|\psi(x_{1})\right|
&\leq\left|\psi\left((p_{\nu}-y\right)x_{1})\right|+\left|\psi(yx_{1})\right|\\
&\leq\sqrt{\psi\left(p_{\nu}-y\right)}
\sqrt{\psi\left(x_{1}^{*}\left(p_{\nu}-y\right)x_{1}\right)}+\left|\psi(yx_{1})\right|\\
&<2\psi(p_{\nu})\sqrt{\epsilon}+\psi(p_{\nu})\epsilon,
\end{align*}
since $p_{\nu}-y\geq0$ and $\psi$ is a positive functional on $p_{\nu}\rg(G)p_{\nu}$.
The above inequality holds for every $\epsilon>0$,
then $\psi(x_{1})=\psi(p_{L}\lambda_{s}p_{L}+p_{L}\lambda_{s^{-1}}p_{L})=0$.
Similarly,
we have $\psi(ip_{L}\lambda_{s}p_{L}-ip_{L}\lambda_{s^{-1}}p_{L})=0$,
then $\psi(\lambda_{s}P_{L})=\psi(p_{L}\lambda_{s}p_{L})=0$.
\end{proof}

\begin{exa}
When $G$ is a \Cs-simple group in Proposition \ref{prop;elementary},
the averaging projection $p_{K_n}$ is a central projection of $\rg(L_{n})$ with $\rg(L_{n})p_{K_n}\cong{\rg(L_{n}/K_{n})}$.
By this isomorphism, the canonical trace of $\rg(L_{n}/K_{n})$ corresponds to the restriction of $E_{K_n}$ on $\rg(L_{n})p_{K_n}$.
Since $L_{n}/K_{n}$ is a discrete \Cs-simple group,
for any $x\in\rg(L_{n})p_{K_n}$ and $\epsilon>0$ there are $g_{1}, g_{2},\dots, g_{m}\in L_{n}$ such that the inequality
\[
\norm{\frac{1}{m}\sum_{j=1}^{m}\lambda_{g_j}(x-E_{K_n}(x))\lambda_{g_j}^*}
=\norm{\frac{1}{m}\sum_{j=1}^{m}\lambda_{g_j}x\lambda_{g_j}^*-E_{K_n}(x)}<\epsilon
\] 
holds, by Theorem 4.5 in \cite{Haa}.

The increasing union $\bigcup_{n}\rg(L_{n})p_{K_n}$ is a dense $\ast$-subalgebra of $\rg(G)$ containing $C_{cc}(G)$ and $G$ is unimodular.
Thus $G$ satisfies the assumptions of Lemma \ref{lem:ce1}.
\end{exa}

\begin{exa}
Suppose that the discrete group $\Lambda$ and the \Cs-simple group $\cG:= \Lambda \rtimes G$ for a totally disconnected group $G$ are defined as in Theorem \ref{prop;non-elementary}.
Let $K$ be a compact open subgroup of $G$,
and $g_{1}, g_{2},\dots, g_{l}$ be elements in $\cG\setminus{K}$.
Since $\cG=\bigcup_{n}\Lambda_{n}\rtimes G$,
we may assume $g_{1}, g_{2},\dots, g_{l}\in\Lambda_{N-1}\rtimes G$ for some $N$.
Each $g_{i}$ has the form $g_{i}=s_{i}h_{i}$,
where $s_{i}$ is in $\Lambda_{N-1}$,
$h_{i}$ is in $G$.
Let $(\delta_{gK})_{gK\in G/K}$ be the canonical generators of $\bigoplus_{G/K} \IZ_2\leq\Upsilon_{N}=\bigoplus_{K\in\Omega}\bigoplus_{G/K} \IZ_2$.
For every $i$,
we have
$\delta_{K}g_{i}\delta_{K}^{-1}=s_{i}\delta_{K}\delta_{h_{i}K}^{-1}h_{i}$ and
$s_{i}^{\prime}:=s_{i}\delta_{K}\delta_{h_{i}K}^{-1}\neq e$
because every $i$ satisfies either $s_{i}\neq e$ or $h_{i}\notin K$.
Since $\Lambda$ has the subgroup $\Gamma_{N}\ast\Xi_{N}\ast\Xi_{N+1}$ and $s_{i}^{\prime}\in\Gamma_{N}$,
for any $\epsilon>0$ there are $t_{1}, t_{2},\dots, t_{m}\in\Xi_{N}\ast\Xi_{N+1}$ such that
$\norm{\frac{1}{m}\sum_{j=1}^{m}\lambda^{\Lambda}_{t_{j}s_{i}^{\prime}t_{j}^{-1}}}<\frac{\epsilon}{l}$ for all $i$
(see Proof of Lemma 1.2 of \cite{P-S}),
where $\lambda^{\Lambda}$ is the left regular representation of $\Lambda$.
The Hilbert space $L^{2}(\cG)$ is isomorphic to $l^{2}(\Lambda)\otimes L^{2}(G)$ and the restriction of $\lambda^{\cG}$ on $\Lambda$ is unitary equivalent to $\lambda^{\Lambda}\otimes 1$.
Since $(t_{j})_{j}$ and $\delta_{K}$ commute with $K$,
they are in the kernel of the modular function.
Moreover, we get
\[\norm{\frac{1}{m}\sum_{j}^{m}\lambda^{\cG}_{t_{j}\delta_{K}}\right(\sum_{i}^{l}p_{K}\lambda^{\cG}_{g_{i}}p_{K}\left)({\lambda_{t_{j}\delta_{K}}^{\cG})^{\ast}}}
=\norm{\sum_{i}^{l}p_{K}(\frac{1}{m}\sum_{j=1}^{m}\lambda^{\cG}_{t_{j}s_{i}^{\prime}t_{j}^{-1}})\lambda^{\cG}_{h_i}p_{K}}<\epsilon\] and
\[p_{K}=\frac{1}{m}\sum_{j}^{m}\lambda^{\cG}_{t_{j}\delta_{K}}p_{K}\lambda^{\cG\ast}_{t_{j}\delta_{K}}.\]
Therefore,
$\cG$ satisfies the assumption of Lemma \ref{lem:ce1} for the net $(K_{\nu})_{\nu}$ of all compact open subgroups of $G$ and the net $(p_{K_\nu}\rg(\cG)p_{K_\nu})_\nu$ of \Cs-subalgebras of $\rg(\cG)$.
\end{exa}

\section{Generalized Baumslag-Solitar groups and their completions}
In this section,
let $(G, Y)$ be a graph of groups with a connected graph $Y$, vertex groups $\{G_{P}\}_{P\in V(Y)}$, edge groups $\{G_{y}\}_{y\in E(Y)}$, and monomorphisms $\{\iota_{y}\colon G_{y}\rightarrow G_{t(y)}\}_{y\in E(Y)}$.
We suppose that $\theta_{y}$ is a canonical isomorphism from $\iota_{y}(G_{y})$ to $\iota_{\overline{y}}(G_{y})$ for any $y\in E(Y)$.
Let $y$ be an edge of $Y$ and $W$ be a subgraph of $Y$ with $E(W):=E(Y)\setminus{\{y,\overline{y}\}}$ and $V(W):=V(Y)$.
Since $Y$ is connected,
$W$ is either connected or decomposed into a disjoint union of two nonempty connected graphs $W_1$ and $W_2$.
\begin{lem}\label{lem:subgraph}
The following statements hold.
\begin{itemize}
\item[(a)] 
If  $W$ is connected,
then $\pi_{1}(G, Y)$ is isomorphic to the HNN extension $\pi_{1}(G|_{W}, W)\ast_{\theta_y}$.
Where $(G|_{W}, W)$ is a graph of groups with edge groups $\{G_{w}\}_{w\in E(W)(=E(Y)\setminus{\{y,\overline{y}\}})}$ and vertex groups $\{G_{P}\}_{P\in V(W)(=V(Y))}$.
\item[(b)]
If $W$ is disconnected,
then $\pi_{1}(G, Y)$ is isomorphic to the amalgamated free product $\pi_{1}(G|_{W_{1}}, W_{1})\ast_{G_{y}}\pi_{1}(G|_{W_{2}}, W_{2})$.
Where $(G|_{W_{i}}, W_{i})$ is a graph of groups with edge groups $\{G_{w}\}_{w\in E(W_{i})(\subset{E(Y)})}$ and vertex groups $\{G_{P}\}_{P\in V(W_{i})(\subset{V(Y)})}$.
\end{itemize}
\end{lem}

\begin{proof}
First we suppose $W$ is connected.
By Lemma 6 of \cite{Se},
there is the natural isomorphism $f\colon F(G, Y)\rightarrow F(G|_{W}, W)\ast_{\theta_{y}}$. 
Put $P_0:=o(y)$ and take a maximal subtree $T$ of $Y$ with $y\notin E(T)$,
then the commutative diagram
\[
\xymatrix{
\pi_{1}(G, Y, P_{0})\ar@<-0.2ex>@{^{(}->}[r]^-{i_{Y}}\ar[d]_-{f|}
&F(G, Y)\ar@{->>}[r]^-{q_{Y}}\ar[d]_-{f}
&\pi_{1}(G, Y, T)\ar[d]_{\overline{f}}\\
\pi_{1}(G|_{W}, W, P_{0})\ast_{\theta_{y}}\ar@<-0.2ex>@{^{(}->}[r]^-{(i_{W})\ast_{\theta_y}}
&F(G|_{W}, W)\ast_{\theta_{y}}\ar@{->>}[r]^-{(q_{W})\ast_{\theta_y}}
&\pi_{1}(G|_{W}, W, T)\ast_{\theta_{y}}
}
\]
holds.
Where $i_{Y}, i_{W}$ are the canonical injections and $q_{Y}, q_{W}$ are the canonical quotient maps as in Proposition \ref{prop;fg}.
By Proposition \ref{prop;fg}, $q_{1}\circ i_{1}$ and $q_{2}\circ i_{2}$ are isomorphism.
Since $f$ is isomorphism,
the induced maps $f|$ and $\overline{f}$ are isomorphism.

Next,
Let $W$ be disconnected.
We define a graph of groups $(H, W^{\prime})$ as follows.
The graph $W^{\prime}$ consists of two edges $\{y, \overline{y}\}$ and two vertices $\{P_{1}:=o(y), P_{2}:=t(y)\}$.
The edge group and the vertex groups of $(H, W^{\prime})$ are  defined as
\[
 H_{y}:=G_{y},
\quad H_{P_{1}}:=F(G|_{W_{1}}, W_{1}),
\quad H_{P_{2}}:=F(G|_{W_{2}}, W_{2}).
\] 
By using Lemma 6 of \cite{Se} repeatedly,
we get the natural isomorphism $g\colon F(G, Y)\rightarrow F(H, W^{\prime})$.
Fix a maximal subtree $T$ of $Y$, 
then $T_{1}:=T\cap Y_{1}$ and $T_{2}:=T\cap Y_{2}$ are the maximal subtrees of $Y_1$ and $Y_2$.
Since $\pi_{1}(H, W^{\prime}, W^{\prime})$ is isomorphic to $F(G|_{W_{1}}, W_{1})\ast_{G_y}F(G|_{W_{2}}, W_{2})$,
there is a commutative diagram
\[
\xymatrix{
F(G, Y)\ar@{->>}[r]^-{q_{Y}}\ar[d]_-{g}
&\pi_{1}(G, Y, T)\ar[rd]^-{\overline{g}}\\
F(H, W^{\prime})\ar@{->>}[r]^-{q_{W^{\prime}}}
&F(G|_{W_{1}}, W_{1})\ast_{G_y}F(G|_{W_{2}}, W_{2})\ar@{->>}[r]^-{q}
&\pi_{1}(G|_{W_{1}}, W_{1}, T_{1})\ast_{G_y}\pi_{1}(G|_{W_{2}}, W_{2}, T_{2})
}
\]
with $q:=q_{W_1}\ast_{G_y}q_{W_2}$.
Define a subgroup $F^{\prime}$ of $F(H, W^{\prime})$ as follows.
\[
F^{\prime}:=\left\{|c, \mu|\in F(H, W^{\prime})\middle|
\begin{array}{l}
(c, \mu)\text{ is a word of }(H, W^{\prime}).\\
c\text{ is a closed  path of }W^{\prime}\text{ begin with }P_{1}.\\
\mu\text{ is a sequence of }\pi_{1}(G|_{W_{1}}, W_{1}, P_{1})\cup\pi_{1}(G|_{W_{2}}, W_{2}, P_{2}).
\end{array}
\right\}
\]
This $F^\prime$ satisfies $g(\pi_{1}(G, Y, P_{1}))\subset{F^{\prime}}$ and the restriction of $q\circ q_{W^\prime}$ on $F^{\prime}$ is an isomorphism.
That is because $F^\prime$ is isomorphic to the fundamental group of the graph of groups $(H^{\prime}, W^{\prime})$ whose vertex groups and edge groups are as follows.
\[
 H^{\prime}_{y}:=G_{y}.
\quad H^{\prime}_{P_{1}}:=\pi_{1}(G|_{W_{1}}, W_{1}, P_{1}).
\quad H^{\prime}_{P_{2}}:=\pi_{1}(G|_{W_{2}}, W_{2}, 2_{2}).
\] 
Since the commutative diagram
\[
\xymatrix{
\pi_{1}(G, Y, P_{0})\ar@<-0.2ex>@{^{(}->}[r]^-{i_{Y}}\ar[d]_-{g|}
&F(G, Y)\ar@{->>}[r]^-{q_{Y}}\ar[d]_-{g}
&\pi_{1}(G, Y, T)\ar[d]_{\overline{g}}\\
F^{\prime}\ar@<-0.2ex>@{^{(}->}[r]
&F(H, W^{\prime})\ar@{->>}[r]^-{q\circ q_{W^\prime}}
&\pi_{1}(G|_{W_{1}}, W_{1}, T_{1})\ast_{G_y}\pi_{1}(G|_{W_{2}}, W_{2}, T_{2})
}
\]
holds,
$\overline{g}$ is an isomorphism.
\end{proof}

\begin{defn}\label{defn}
The fundamental group $\pi_{1}(G, Y)$ is called a generalized Baumslag-Solitar group if the edge groups $\{G_{y}\}_{y\in E(Y)}$ and the vertex groups $\{G_{P}\}_{P\in V(Y)}$ are isomorphic to $\IZ$.
\end{defn}

Suppose that $(G, Y)$ is a graph of groups whose edge groups and vertex groups are isomorphic to $\IZ$ and $Y$ consists of one vertex $P$ and two edges $y, \overline{y}$.
If $[G_{P}:\iota_{\overline{y}}(G_{y})]=n$ and $[G_{o(y)}:\iota_{y}(G_{y})]=m$,
then the generalized Baumslag-Solitar group $\pi_{1}(G, Y)$ is isomorphic to the Baumslag-Solitar group BS$(n, m)$.

Let $(G, Y)$ be a general graph of groups whose edge groups and vertex groups are isomorphic to $\IZ$.
Suppose $T$ is a maximal subtree of $Y$.
For a geodesic path $c\subset{T}$ from $P$ to $Q$,
we define the integer $k_{c}$ as follows.
Suppose that $c:=(y_{1}, y_{2},\dots, y_{l})$, $[G_{o(y_{i})}:\iota_{\overline{y_{i}}}(G_{y_{i}})]=n_{i}$ and $[G_{t(y_{i})}:\iota_{y_{i}}(G_{y_{i}})]=m_{i}$.
Set
$k_{1}:=m_{1}$ and $k_{i}:=\frac{k_{i-1}m_{i}}{(k_{i-1}, n_{i})}$
for any $1\leq i\leq l$.
Where $(k_{i-1}, n_{i})$ is the greatest common divisor of $k_{i-1}$ and $n_{i}$.
We put $k_{c}:=k_{l}$ for a geodesic path $c$ with the length $l\neq0$ and $k_{c}:=1$ for a path $c$ with   the length $0$.
By a direct computation,
we get $[G_{t(y_{l})}:G_{t(y_{l})}\cap G_{o(y_{1})}]=k_{c}$ and $[G_{o(y_{1})}:G_{t(y_{l})}\cap G_{o(y_{1})}]=k_{\overline{c}}$ in $\pi_{1}(G, Y, T)$.
Where $\overline{c}:=(\overline{y}_{l}, \overline{y}_{l-1},\dots, \overline{y}_{1})$.

In addition,
if $y$ is an edge of $Y$ with $[G_{o(y)}:\iota_{\overline{y}}(G_{y})]=n$ and $[G_{t(y)}:\iota_{y}(G_{y})]=m$,
then there is the geodesic path $c\subset{T}$ from $o(y)$ to $t(y)$.
The relative order $[G_{o(y)}:\iota_{\overline{y}}(G_{y})\cap\iota_{y}(G_{y})]$ in $\pi_{1}(G, Y, T)$ can be compute as the least common multiple $k^{\prime}_{\overline{y}}$ of $n$ and $\frac{k_{\overline{c}}m}{(m, k_{c})}$.
We define $\kappa_{\overline{y}}:=\frac{k^{\prime}_{\overline{y}}}{n}$,
then we have $[\iota_{\overline{y}}(G_{y}):\iota_{\overline{y}}(G_{y})\cap\iota_{y}(G_{y})]=\kappa_{\overline{y}}$ in $\pi_{1}(G, Y, T)$.

By using above notations, we state our main theorem as follows.

\begin{thm}\label{prop:GBS}
Let $T$ be a maximal subtree of $Y$ and $\overline{\pi_{1}(G, Y, T)}$ be the closure of $\pi_{1}(G, Y, T)$ 
in the automorphism group of the universal covering $\tilde{X}(G, Y, T)$.
If $(G, Y, T)$ satisfies the following conditions,
then $\overline{\pi_{1}(G, Y, T)}$ is a non-discrete locally compact \Cs-simple group. 
\begin{itemize}
\item The graph $Y$ is not a tree.
\item Every edge groups and vertex groups of $(G, Y)$ are isomorphic to $\IZ$.
\item There is an edge $y\in E(Y)\setminus E(T)$ such that $\kappa_{\overline{y}}\neq\kappa_{y}$.
\item For any edge $z$ of $Y$, $\iota_{z}(G_{z})$ is a proper subgroup of $G_{t(z)}$.
\end{itemize}
Moreover, the reduced group \Cs-algebra $\rg(\overline{\pi_{1}(G, Y, T)})$ has a unique KMS-weight with respect to the modular flow.
\end{thm}

In the case of Baumslag-Solitar groups, the following are known.
(See \cite{R2}.)
\begin{itemize}
\item[$\circ$] $\overline{\text{BS}(n, m)}$ is discrete if and only if $|n|=|m|$.
\item[$\circ$] $\overline{\text{BS}(n, m)}$ is \Cs-simple if and only if $|n|, |m|\neq1$.
\end{itemize}
In Theorem \ref{prop:GBS}, the third condition corresponds to $|n|\neq|m|$ and the fourth condition corresponds to $|n|, |m|\neq1$.

From now on,
we fix the graph of groups $(G,Y)$, the maximal subtree $T$ of $Y$ and $y\in E(Y)\setminus E(T)$ which satisfy the assumptions in Theorem \ref{prop:GBS}.
In order to prove Theorem \ref{prop:GBS},
we prepare some lemmas as follows.

\begin{lem}\label{lem:faithful}
The action of $\pi_{1}:=\pi_{1}(G, Y, T)$ on the universal covering $\tilde{X}:=\tilde{X}(G, Y, T)$ is faithful.
\end{lem}

\begin{proof}
Suppose that $g\in\pi_{1}$ acts trivially on $\tilde{X}$.
It suffices to show $g=e$.
Since the left multiplication of $g$ preserves all elements of $V(\tilde{X})=\sqcup\pi_{1}/G_{P}$,
we have $g_{y}^{n}gg_{y}^{-n}\in G_{o(y)}\cap G_{t(y)}$ for all $n\in\IZ_{\geq0}$.
Let $W$ be a connected subgraph of $Y$ with $V(W):=V(Y)$, $E(W):=E(Y)\setminus\{y, \overline{y}\}$.
By Lemma \ref{lem:subgraph},
$\pi_{1}$ has a structure of HNN extension and $g_{y}$ corresponds to the stable letter.
Therefore, we get
\[
g\in G_{t(y)}\quad and \quad g_{y}^{n}gg_{y}^{-n}\in\iota_{y}(G_{y})\cap\iota_{\overline{y}}(G_y)
\]
for every $n\in\IZ_{\geq0}$.
Since all edge groups of $Y$ are isomorphic to $\IZ$,
we can take the free generators $a$ of $\iota_{y}(G_{y})$ and $b$ of $\iota_{\overline{y}}(G_y)$ with $g_{y}ag_{y}^{-1}=b$.
Take the integers $N, M\in\IZ_{>0}$ with $\iota_{y}(G_y)\cap\iota_{\overline{y}}(G_y)=\langle a^{N}\rangle=\langle b^{M}\rangle$,
then $N\neq M$ holds because of the assumption $\kappa_{y}\neq\kappa_{\overline{y}}$.
By replacing $y$ by $\overline{y}$,
we may assume $\left|\frac{N}{(N, M)}\right|\neq1$.
By straightforward computation, we have 
\[\langle a\rangle\cap g_{y}^{-n}(\iota_{y}(G_{y})\cap\iota_{\overline{y}}(G_y))g_{y}^{n}
=\left\langle a^{(N, M)\left|\frac{M}{(N, M)}\right|^{n}}\right\rangle
\] 
Thus $g=e$.
\end{proof}

From now on, we put $\pi_{1}:=\pi_{1}(G, Y, T)$ and $\tilde{X}:=\tilde{X}(G, Y, T)$ as used in Lemma \ref{lem:faithful}.
Let $\textbf{Aut}(\tilde{X})$ be the automorphism groups of $\tilde{X}$.
We suppose that $\overline{\Gamma}$ is the closure of $\Gamma$ in $\textbf{Aut}(\tilde{X})$ for any subset $\Gamma\subset\pi_{1}$. 
For $g\in\pi_{1}$ and $P\in V(Y)$, the subgroup of $\textbf{Aut}(\tilde{X})$ which stabilize $gG_{P}\in V(\tilde{X})$ is denoted by $\textbf{Stab}(gG_{P})$.
If $w\in E(Y)$,
then there is a subgraph $W$ of $Y$ with $V(W)=V(Y)$ and $E(W)=E(Y)\setminus{y, \overline{y}}$.
We use the letter $Y\setminus{\{y\}}$ for this $W$.

\begin{lem}\label{lem:local basis}
For any $P\in Y$,
the family of finite intersections of $\{\overline{\pi}_{1}\cap\rm{\mathbf{Stab}}(gG_{P})\}_{g\in \pi_{1}}$ gives a neighborhood basis of $e$ in $\overline{\pi}_{1}$.
\end{lem}

\begin{proof}
By definition of $\overline{\pi}_{1}$,
the family of finite intersections of the sets $\{\overline{\pi}_{1}\cap\textbf{Stab}(gG_{P})\mid|g\in \pi_{1}, Q\in V(Y)\}$ is a neighborhood basis of $e$.
Thus it suffices to show that for any $g\in \pi_{1}$ and $Q\in V(Y)$,
there are $g_{1}, g_{2},\dots, g_{n}\in\pi_{1}$ with $\bigcap_{i=1}^{n}\overline{\pi}_{1}\cap\textbf{Stab}(g_{i}G_{P})\subset{\overline{\pi}_{1}\cap\textbf{Stab}(gG_{P})}$.

First,
we assume there is an edge $w$ such that $o(w)=P$ and $t(w)=Q$.
Fix $g\in\pi_{1}$.
When the graph $Y\setminus{\{w\}}$ is disconnected,
there are subgraphs $W_{1}$ and $W_{2}$ with $\pi_{1}(G|_{W_{1}}, W_{1})\ast_{G_{w}}\pi_{1}(G|_{W_{2}}, W_{2})\cong\pi_{1}$ as in Lemma\ref{lem:subgraph}.
Suppose $P=o(w)\in V(W_{1})$ and $Q=t(w)\in V(W_{2})$.
We can take $b\in G_{Q}\setminus{\iota_{w}(G_{w})}\subset{\pi_{1}(G|_{W_{2}}, W_{2})}$.
For any $h\in\pi_{1}$,
if $g^{-1}hg\in G_{P}$ and $b^{-1}g^{-1}hgb\in G_{P}$ hold,
then $g^{-1}hg\in G_{y}\subset{G_Q}$.
Thus we have $\pi_{1}\cap\textbf{Stab}(gG_{P})\cap\textbf{Stab}(gbG_{Q})\subset\pi_{1}\cap\textbf{Stab}(gG_{Q})$.
By taking closures, we get $\overline{\pi}_{1}\cap\textbf{Stab}(gG_{P})\cap\textbf{Stab}(gbG_{P})\subset\overline{\pi}_{1}\cap\textbf{Stab}(gG_{Q})$ since stabilizers are clopen subgroups of $\textbf{Aut}(\tilde{X})$.

When the graph $W=Y\setminus{\{w\}}$ is disconnected, 
we have $\pi_{1}(G|_{W}, W)\ast_{\theta_w}\cong\pi_{1}$.
Thus for any $h\in\pi_{1}$,
if $g^{-1}hg\in G_{P}$ and $g_{w}g^{-1}hgg_{w}^{-1}\in G_{P}$ hold,
then $g^{-1}hg\in \iota_{w}(G_{w})$.
We get $\pi_{1}\cap\textbf{Stab}(gG_{P})\cap\textbf{Stab}(gg_{w}^{-1}G_{Q})\subset\pi_{1}\cap\textbf{Stab}(gG_{Q})$.
This leads to $\overline{\pi}_{1}\cap\textbf{Stab}(gG_{P})\cap\textbf{Stab}(gbG_{P})\subset\overline{\pi}_{1}\cap\textbf{Stab}(gG_{Q})$.

For a general $Q$,
there is a geodesic path from $P$ to $Q$ in $Y$.
By using same argument repeatedly,
we get $g_{1}, g_{2},\dots, g_{n}\in\pi_{1}$ with $\bigcap_{i=1}^{n}\overline{\pi}_{1}\cap\textbf{Stab}(g_{i}G_{P})\subset{\overline{\pi}_{1}\cap\textbf{Stab}(gG_{P})}$.
\end{proof}

\begin{lem}\label{lem:loc cpt}
For any $P\in V(Y)$,
the closure $\overline{G}_{P}$ is a non-discrete compact open subgroup of $\overline{\pi}_{1}$.
\end{lem}
\begin{proof}
Since $\overline{G}_{P}=\overline{\pi}_{1}\cap\textbf{Stab}(G_{P})$,
it is trivial that $\overline{G}_{P}$ is open.

By Lemma \ref{lem:faithful} and \ref{lem:local basis},
we have $\bigcap_{g\in\pi_{1}}\textbf{Stab}(gG_{P})\cap\pi_{1}=\{e\}$ for any $P\in V(Y)$.
Let $S_{\pi_{1}/G_{P}}$ be the topological group of all bijections from $\pi_{1}/G_{P}$ to $\pi_{1}/G_{P}$. There is the monomorphism $f\colon\pi_{1}\rightarrow S_{\pi_{1}/G_{P}}$ given by left multiplication.
This $f$ is continuous with respect to the relative topology of $\pi_{1}$ induced by $\textbf{Aut}(\tilde{X})$
and $S_{\pi_{1}/G_{P}}$ is complete.
Thus $f$ extends to a continuous homomorphism from $\overline{\pi}_{1}$.
By Lemma \ref{lem:local basis},
the extension of $f$ gives a homeomorphism from $\overline{\pi}_{1}$ onto $\overline{f(\pi_{1})}$.
Therefore,
it suffices to show that $\overline{f(G_{P})}\subset{S_{\pi_{1}/G_{P}}}$ is compact.

Let $a$ be a generator of $G_{P}$ and $g\in\pi_{1}$.
There is a word $(c,\mu)$ of $(G,Y)$ such that $c$ is a closed path from $P$ to $P$ and $g=|c, \mu|$.
Suppose that $c=(y_{1}, y_{2},\dots, y_{n})$ and $m_{i}:=[G_{t(y_{i})}:\iota_{y_{i}}(G_{y_{i}})]$.
Since all vertex group of $(G, Y)$ is isomrphic to $\IZ$,
we get $ga^{(\prod_{i=1}^{n}m_{i})}g^{-1}\in G_{P}\setminus\{e\}$.
Thus for any $g\in\pi_{1}$,
the relative order $[G_{P}:G_{P}\cap gG_{P}g^{-1}]$ is finite.
By using the argument in Section 4 of \cite{Tz},
the closure $\overline{f(G_{P})}\subset{S_{\pi_{1}/G_{P}}}$ is compact.

In proof of the compactness of $\overline{G}_{P}$,
we have $\textbf{Stab}(gG_{P})\cap G_{P}=gG_{P}g^{-1}\cap G_{P}\neq\{e\}$ for any $g\in\pi_{1}$.
Since $G_{P}$ is isomorphic to $\IZ$,
the intersection $\bigcap_{i=1}^{n}\textbf{Stab}(g_{i}G_{P})\cap\overline{G_{P}}$ is not a trivial subgroup of $\overline{G_{P}}$ for any $g_{1}, g_{2},\dots, g_{n}\in\pi_{1}$.
Thus $\{e\}$ is not an open subgroup of $\overline{G_{P}}$ and $\overline{G_{P}}$ is non-discrete.
\end{proof}
This lemma shows that $\overline{\pi}_{1}$ is a non-discrete locally compact group.
To prove \Cs-simplicity of $\overline{\pi}_{1}$,
we use the following lemma which is shown in a similar way to the proof of Lemma 1.2 in \cite{P-S}.

\begin{lem}\label{lem:PS}
Let $G$ be a totally disconnected group and $K$ be a compact open subgroup.
If any finite set $\{g_{1}, g_{2},\dots, g_{l}\}$ of $G\setminus{K}$ has pairwise disjoint subsets  $S_{1}, S_{2},\dots, S_{9}$ of $G$ and $z_{1}, z_{2},\dots, z_{9}\in G$ with the following conditions,
then property (2) in Lemma \ref{lem:ce1} holds for $A_{\nu}=\rg(G)$ and $K_{\nu}=K$.
\begin{itemize}
\item Each $S_{i}$ satisfies $KS_{i}=S_{i}$.
\item 
Each  $z_i$ commutes with every element of K and satisfies $z_{i}g_{j}z_{i}^{-1}(G\setminus{S_{i}})\subset{S_{i}}$ for every $g_{j}$.
\end{itemize}
\end{lem}

\begin{proof}
Take a self-adjoint element $f\in C_{c}(G\setminus{K})\subset{\rg(G)}$,
then there are $g_{1}, g_{2},\dots, g_{l}\in G\setminus{K}$ which have pairwise distinct $K$-double cosets and $f$ is supported on $\bigcup_{j=1}^{l}Kg_{j}K$.
We have subsets $S_{1}, S_{2},\dots, S_{9}$ of $G$ and $z_{1}, z_{2},\dots, z_{9}\in G$ which satisfy the assumptions of the lemma for $g_{1}, g_{2},\dots, g_{l}$. 
Since $S_{1}, S_{2},\dots, S_{9}$ are unions of $K$-double cosets,
they are clopen.
We define $q_{i}$ as the orthogonal projection onto $L^{2}(S_{i})$,
then for any $i$ and $g\in\bigcup_{j=1}^{l}Kg_{j}K$,
we get $(1-q_{i})\lambda_{z_{i}gz_{i}^{-1}}(1-q_{i})=0$.
Thus  the self-adjoint element $f$ satisfies $(1-q_{i})\lambda_{z_{i}}f\lambda_{z_i}^{*}(1-q_{i})=0$.
Generally,
if a self-adjoint operator $T$ and a projection $q$ on a Hilbert space $\cH$ satisfy $(1-q)T(1-q)=0$,
then $\left|\langle T\xi,\xi\rangle\right|\leq2\|T\|\|q\xi\|$ holds for any unit vectors $\xi\in\cH$.
Thus we have
\begin{align*}
|\langle(\frac{1}{9}\sum_{i=1}^{9}\lambda_{z_{i}}f\lambda_{z_i}^{*})\xi,\xi\rangle|
\leq\frac{2}{3}\|f\|,
\end{align*}
for any unit vector $\xi$ (see proof of Lemma 1.2 of \cite{P-S}),
then 
\[\left\|\frac{1}{9}\sum_{i=1}^{9}\lambda_{z_{i}}f\lambda_{z_i}^{*}\right\|\leq\frac{2}{3}\|f\|.\]
Since the continuous function $\frac{1}{9}\sum_{i=1}^{9}\lambda_{z_{i}}f\lambda_{z_i}^{*}$ is supported on $\bigcup_{i, j}Kz_{i}g_{j}z_{i}^{-1}K$ and the assumptions $z_{i}g_{j}z_{i}^{-1}(G\setminus{A_{i}})\subset{A_{i}}$ hold for every $i, j$,
the support of $\frac{1}{9}\sum_{i=1}^{9}\lambda_{z_{i}}f\lambda_{z_i}^{*}$ is on $G\setminus{K}$.
By repeating the same process,
for any $\epsilon>0$,
we can take  $w_{1}, w_{2}, \dots,w_{n}\in G$ with $\|\frac{1}{n}\sum_{i=1}^{n}\lambda_{w_{i}}f\lambda_{w_i}^{*}\|<\epsilon$ and $\lambda_{w_{i}}p_{K}\lambda_{w_{i}}^{*}=p_{K}$.
Therefore condition (2) in Lemma \ref{lem:ce1} holds for $A_{\nu}=\rg(G)$ and $K_{\nu}=K$.
\end{proof}

In the following lemma,
we focus on reduced forms of elements in groups of HNN extension.
Let $G\ast_{\theta}$ be a HNN extension with respect to a discrete group $G$ and an isomorphism $\theta$ between subgroups $L_{1}$ and $L_{2}$ of $G$.
Suppose that the stable letter of $G\ast_{\theta}$ is denoted by $t$.
If $g\in G\ast_{\theta}$ has two reduced forms 
\[
c_{0}t^{\epsilon_1}c_{1}t^{\epsilon_2}c_{2}\cdots t^{\epsilon_{l}}c_{l}
\quad\text{and}\quad
d_{0}t^{\epsilon_{1}^{\prime}}d_{1}t^{\epsilon_{2}^{\prime}}d_{2}\cdots t^{\epsilon_{m}^{\prime}}d_{m}, 
\]
then we have $l=m$ and $(\epsilon_{1}, \epsilon_{2},\dots, \epsilon_{l})=(\epsilon_{1}^{\prime}, \epsilon_{2}^{\prime},\dots, \epsilon_{m}^{\prime})$.
(See Lemma 2.3 of \cite{L-S}.)
We call this $l=m$ the length of $g$.

\begin{lem}\label{lem:ce2}
There is a compact open subgroup $K$ of $\overline{\pi}_{1}$ which is contained in $\overline{G_{o(y)}}$ and satisfies the conditions of Lemma \ref{lem:PS}.
\end{lem}

\begin{proof}
Let $a, b$ be free generators of $G_{t(y)}$ and $G_{o(y)}$.
We suppose that $\iota_{y}(G_{y})=\langle a^{n}\rangle$, $\iota_{\overline{y}}(G_{y})=\langle b^{m}\rangle$, and $g_{y}a^{n}g_{y}^{-1}=b^{m}$ for some $n, m\in\IZ$.
We have $|n|, |m|\geq2$ by assumption.
If $\gamma$ is a geodesic path in $T$ from $o(y)$ to $t(y)$ and 
$H$ is a subgroup $G_{t(y)}\cap G_{o(y)}=\langle a^{k_{\gamma}}\rangle=\langle b^{k_{\overline{\gamma}}}\rangle$ of $\pi_{1}$,
then the subgroup $\langle a, b\rangle\subset\pi_{1}$ is isomorphic to $\langle a\rangle\ast_{H}\langle b\rangle$.
Define $N:=\frac{n^{2}k_{\overline{\gamma}}}{(n, \frac{k_{\gamma}m}{(m, k_{\overline{\gamma}})})(m, k_{\overline{\gamma}})}$ and a compact open subset $K:=\overline{\langle a^{N}\rangle}\subset\overline{\pi}_{1}$.
When we take a finite set $\{g_{1}^{\prime}, g_{2}^{\prime},\dots,g_{s}^{\prime}\}$ in $\overline{\pi}_{1}\setminus K$,
there is $g_{i}\in\pi_{1}\setminus\langle a^{N}\rangle$ with $g_{i}^{\prime}\in g_{i}K$ for any $i$,
since $\overline{\pi}_{1}\setminus{K}\subset(\pi_{1}\setminus\langle a^{N}\rangle)K$.
Let $W$ be a subgraph $Y\setminus{\{y\}}$ of $Y$.
By the isomorphism $\pi_{1}(G|_{W}, W)\ast_{\theta_y}\cong\pi_{1}$,
$\pi_{1}$ has the structure of HNN extension and $t:=g_{y}$ corresponds to the stable letter.
We have $\langle a^{N}\rangle=\langle a\rangle\cap t^{-2}\langle b\rangle t^{2}$ by definition of $N$.
Let $l_{y}\colon\pi_{1}\rightarrow\IZ_{\geq0}$ be the function such that $l_{y}(g)$ is the length of $g$ as an element of the HNN extension $\pi_{1}(G|_{W}, W)\ast_{\theta_y}$.
Suppose that $L:=\text{max}\{l_{y}(g_{i})\mid1\leq i\leq s\}$.

Set $r_{1}:=at^{-1}bt$.
For every $g\in\pi_{1}$ with $l_{y}(g)\leq L$,
We claim that if $g\notin\bigcup_{i\in\IZ}r_{1}^{i}\langle a^{n}\rangle$,
then $r_{1}^{L}gr_{1}^{-L}$ has a reduced form which begins with $at^{-1}$ and ends with $ta^{-1}$.
Suppose $g\notin\bigcup_{i\in\IZ}r_{1}^{i}\langle a^{n}\rangle$ has the reduced form $c_{0}t^{\epsilon_1}c_{1}t^{\epsilon_2}c_{2}\cdots t^{\epsilon_{l_{y}(g)}}c_{l_{y}(g)}$. 
When $(\epsilon_{1}, \epsilon_{2},\dots, \epsilon_{l_{y}(g)})\neq(-1, 1, -1, 1,\dots, -1, 1)$,
the claim is trivial.
Thus we may assume $(\epsilon_{1}, \epsilon_{2},\dots, \epsilon_{l_{y}(g)})=(-1, 1, -1, 1,\dots, -1, 1)$.
First,
we assume $l_{y}(gr_{1}^{-L})\geq l_{y}(r_{1}^{L}g)=2L+l_{y}(g)-2M$.
When $M=0$,
the words $at^{-1}btc_{0}t^{\epsilon_1}c_{1}t^{\epsilon_2}c_{2}\cdots t^{\epsilon_{l_{y}(g)}}c_{l_{y}(g)}t^{-1}b^{-1}ta^{-1}$ are reduced and the claim is true.
When $M\neq0$,
we get a reduced form of $r_{1}^{L}g$ as 
\[
 \begin{cases}
 r_{1}^{(L-\frac{M}{2})}c_{M}^{\prime}t^{-1}c_{M+1}tc_{M+2}t^{-1}\cdots tc_{l(g)}&{\rm~if~}M{\rm~is~even},\\
  r_{1}^{(L-\frac{M+1}{2})}at^{-1}c_{M}^{\prime}tc_{M+1}t^{-1}c_{M+2}t\cdots tc_{l(g)}&{\rm~if~}M{\rm~is~odd}.
 \end{cases}
\]
Here we have
\[
\begin{cases}
{c_{M}^{\prime}\in\langle a^{n}\rangle ac_{M}}{\rm~and~}{c_{M}\notin\langle a^{n}\rangle}&{\rm~if~}M{\rm~is~even},\\
{c_{M}^{\prime}\in\langle b^{m}\rangle bc_{M}}{\rm~and~}{c_{M}\notin\langle b^{m}\rangle}&{\rm~if~}M{\rm~is~odd}.
 \end{cases}
\]
This condition holds even if $M=l(g)$ because of the assumption $g\notin\bigcup_{i\in\IZ}r_{1}^{i}\langle a^{n}\rangle$.
Thus we get
\[
\begin{cases}
{c_{M}^{\prime}a^{-1}\notin\langle a^{n}\rangle}&{\rm~if~}M{\rm~is~even},\\
{c_{M}^{\prime}b^{-1}\notin\langle b^{m}\rangle}&{\rm~if~}M{\rm~is~odd}.
 \end{cases}
\]
Therefore,
we have a reduced form of $r_{1}^{L}gr_{1}^{-L}$ whose length is longer than $|l_{y}(r_{1}^{L}g)-l_{y}(r_{1}^{-L})|$.
Thus the claim is true when $l_{y}(gr_{1}^{-L})\geq l_{y}(r_{1}^{L}g)$.
Similarly, we can prove the claim when $l_{y}(gr_{1}^{-L})\leq l_{y}(r_{1}^{L}g)$.

For $j\in\IZ_{>0}$,
we define $S_{j}^{\prime}$ as the set of all $g\in\pi_{1}$ with a reduced form $g=c_{0}t^{\epsilon_1}c_{1}t^{\epsilon_2}c_{2}\cdots t^{\epsilon_{l(g)}}c_{l(g)}$ such that
\[
(\epsilon_{2i-1}, \epsilon_{2i})=(-1, 1)\quad\text{for every $i\leq j$,}\quad\text{and}\quad(\epsilon_{2j+1}, \epsilon_{2j+2})=(-1, -1).
\]
Suppose $r_{2}:=at^{-2}bt^{2}$ and $z_{j}:=r_{1}^{j}r_{2}r_{1}^{L}$.
These $\{z_{j}\}_{j}$ commute with $a^{N}$ and satisfy the following.
The inclusions
\[
z_{j}gz_{j}^{-1}(\pi_{1}\setminus{S_{j}^{\prime}})\subset{S_{j}^{\prime}}
\]
hold for every $g\in\pi_{1}\setminus\langle a^{N}\rangle$ with $l_{y}(g)\leq L$.

That is because 
$r_{2}r_{1}^{L}gr_{1}^{-L}r_{2}^{-1}$ has a reduced form which begins with $at^{-2}$ and ends with $t^{2}a^{-1}$, when $g$ satisfies either $l(g)\leq L$ and $g\notin\bigcup_{i\in\IZ}r_{1}^{i}\langle a^{n}\rangle$ or $g\in\bigcup_{i\in\IZ}r_{1}^{i}\langle a^{n}\rangle\setminus\langle a^{N}\rangle$.
Thus if $f\in\pi_{1}\setminus{S_{j}^{\prime}}$,
then $z_{j}gz_{j}^{-1}f$ has a reduced form beginning with the words $r_{1}^{j}at^{-2}$.

Since $\overline{S_{j}^{\prime}}=\overline{\langle a^{N}\rangle}S_{j}^{\prime}$,
the sets $\{S_{j}:=\overline{S_{j}^{\prime}}\}_{j}$ are pairwise disjoint subsets of $\overline{\pi}_{1}$ satisfying $KS_{j}=S_{j}$.
The subsets $S_{1}, S_{2},\dots, S_{9}$ of $\overline{\pi}_{1}$ and $z_{1}, z_{2},\dots, z_{9}\in \overline{\pi}_{1}$ satisfy conditions of Lemma \ref{lem:PS} for $\{g_{1}^{\prime}, g_{2}^{\prime},\dots,g_{s}^{\prime}\}$.
\end{proof}

In the proof of the above lemma,
both $r_{1}$ and $r_{2}$ are in the kernel of the modular function of $\overline{\pi}_{1}$,
since $a, b$ are elements in compact open subgroups of $\overline{\pi}_{1}$.
Thus $A_{\nu}:=\rg(\overline{\pi}_{1})$ and $E_{\nu}:=E_{K}$ satisfy condition (2) of Lemma \ref{lem:ce1} and every $g_{i}$ can be taken in the kernel of the modular function.

\begin{proof}[Proof of Proposition \ref{prop:GBS}]
Take $g\in\pi_{1}$.
We show that there is a compact open subgroup $K_{0}\subset\overline{G}_{t(y)}\cap\textbf{Stab}(gG_{t(y)})$ such that  $A_{\nu}:=\rg(\overline{\pi}_{1})$ and $E_{\nu}:=E_{K_{0}}$ satisfy condition (2) of Lemma \ref{lem:ce1}.
We assume that the following statements hold as in the proof of Lemma \ref{lem:ce2}.

\begin{itemize}
\item[$\circ$] The elements $a, b\in\pi_{1}$ are free generators of $G_{t(y)}$ and $G_{o(y)}$.
\item[$\circ$] The integers $n, m$ satisfy $\iota_{y}(G_{y})=\langle a^{n}\rangle$, $\iota_{\overline{y}}(G_{y})=\langle b^{m}\rangle$, and $ta^{n}t^{-1}=b^{m}$, where $t:=g_{y}$.
\item[$\circ$] The function $l_{y}\colon\pi_{1}\rightarrow\IZ_{\geq0}$ gives lengths of elements of $\pi_{1}$ with respect to the identification $\pi_{1}(G|_{W}, W)\ast_{\theta_y}\cong\pi_{1}$ of $\pi_{1}$.
\end{itemize}
 
Since $a\notin\langle b^{m}\rangle$,
there are $c, d\in\{a, e\}$ such that the equation $l_{y}(t^{\pm1}cgdt^{\pm1})=l_{y}(g)+2$ holds.
Set $r^{\prime}_{j}:=gdt^{-j}bt^{j}dg^{-1}ct^{-j}bt^{j}c$ for $j\in\IZ$ and $K_{1}:=\overline{\langle a\rangle}\cap\textbf{Stab}(r^{\prime}_{2}\langle a\rangle)$.
By assumptions for $c, d$,
if $h\in\langle a\rangle\cap\textbf{Stab}(r^{\prime}_{2}\langle a\rangle)$,
then we have
\[
g^{-1}hg\in\langle a\rangle,
\quad(t^{-2}bt^{2})^{-1}g^{-1}hg(t^{-2}bt^{2})\in\langle a\rangle,
\quad{\rm and}\quad(t^{-2}bt^{2})^{-1}h(t^{-2}bt^{2})\in\langle a\rangle.
\]
That is because $(t^{-2}bt^{2})^{-1}h(t^{-2}bt^{2})=h$ holds for any $h\in\langle a\rangle\cap\textbf{Stab}(t^{-2}bt^{2}\langle a\rangle)$.
We get $\langle a\rangle\cap\textbf{Stab}(r^{\prime}_{2}\langle a\rangle)
=\langle a\rangle\cap\textbf{Stab}(g\langle a\rangle)\cap\textbf{Stab}(t^{-2}bt^{2}\langle a\rangle)
\cap g\textbf{Stab}(t^{-2}bt^{2}\langle a\rangle)g^{-1}$.
Thus we have $K_{1}\subset\overline{\langle a\rangle}\cap\textbf{Stab}(gG_{t(y)})$ and ${r_{2}^{\prime}}^{-1}hr_{2}^{\prime}=h={r_{1}^{\prime}}^{-1}hr_{1}^{\prime}$  for any $h\in K_{1}$.

For any $f\in \pi_{1}$ with $l_{y}(f)\geq i$,
we define $\epsilon(f, i)$ as a sequence $(\epsilon_{1}, \epsilon_{2},\dots, \epsilon_{i})$ of $\{0, 1\}$,
when $f$ has a reduced form $c_{0}t^{\epsilon_1}c_{1}t^{\epsilon_2}c_{2}\dots t^{\epsilon_{l(f)}}c_{l(f)}$.
By using this definition, we set
\begin{align*}
w_{i}&:={r_{1}^{\prime}}^{-i}{r_{2}^{\prime}}^{-2},\\
U_{i}^{\prime}&:=\{f\in\pi_{1}\mid l_{y}(f)\geq il_{y}(r_{1}^{\prime})+2{\rm~and~}
\epsilon(f, il_{y}(r_{1}^{\prime})+2)=\epsilon(w_{i}, il_{y}(r_{1}^{\prime})+2)\},\\
U_{i}&:=\overline{U_{i}^{\prime}}=K_{1}U_{i}^{\prime}.\\
\end{align*}
This $\{U_{i}\}_{i=1}^{\infty}$ is an infinite family of mutually disjoint clopen subsets of $\overline{\pi}_{1}$ and we have $w_{i}hw_{i}^{-1}(\overline{\pi}_{1}\setminus{U_{i}})\subset{U_{i}}$ for any $h\in\overline{\langle a\rangle}\setminus{K_{1}}$.
In the same way as in the proof of Lemma \ref{lem:PS},
we get 
\[\left\|\frac{1}{N}\sum_{i=1}^{N}\lambda_{w_{i}}(x-E_{K_1}(x))\lambda_{w_i}^{*}\right\|\leq\frac{2}{\sqrt{N}}\left\|x-E_{K_{1}}(x)\right\|\]
for every self-adjoint element $x\in\rg(\overline{\langle a\rangle})\subset\rg(\overline{\pi_{1}})$.
Let $K\subset\overline{\langle a\rangle}$ be a compact open subgroup of Lemma \ref{lem:ce2} and $K_{0}:=K\cap K_{1}$,
then $E_{K_{0}}=E_{K_{1}}\circ E_{K}$.
For any self-adjoint element $x\in\rg(\overline{\pi}_{1})$ and $\delta>0$,
there are $N\in\IZ_{\geq0}$ and $g_{1}, g_{2},\dots, g_{M}\in\pi_{1}$
such that they satisfy
\[\left\|\frac{1}{M}\sum_{i=1}^{M}\lambda_{g_{i}}(x-E_{K}(x))\lambda_{g_i}^{*}\right\|\leq\frac{\delta}{2}\quad
{\rm and}\quad
\left\|\frac{1}{N}\sum_{i=1}^{N}\lambda_{w_{i}}(E_{K}(x)-E_{K_1}\circ E_{K}(x))\lambda_{w_i}^{*}\right\|\leq\frac{\delta}{2}.
\]
By the proof of Lemma \ref{lem:ce2},
we may assume $\{g_{i}\}_{i}$  are in the kernel of the modular function  and $\{\lambda_{g_{i}}\}_{i}$ commute with all elements of $\rg(K)$.
Thus we have
\begin{align*}
&\quad\left\|\frac{1}{NM}\sum_{i, j}\lambda_{w_{j}g_{i}}(x-E_{K_{0}}(x))\lambda_{w_{j}g_i}^{*}\right\|\\
&\leq
\left\|\frac{1}{N}\sum_{j=1}^{N}\lambda_{w_{j}}(E_{K}(x)-E_{K_1}\circ E_{K}(x))\lambda_{w_j}^{*}\right\|+
\left\|\frac{1}{N}\sum_{j=1}^{N}\lambda_{w_{j}}\left(\frac{1}{M}\sum_{i=1}^{N}\lambda_{g_{i}}(x-E_{K}(x))\lambda_{g_i}^{*}\right)\lambda_{w_j}^{*}\right\|\\
&\leq\delta.
\end{align*}
Moreover,
every $w_{j}g_{i}$ is in the kernel of the modular function and they satisfiy 
\[\frac{1}{NM}\sum_{i, j}\lambda_{w_{j}g_{i}}p_{K_0}\lambda_{w_{j}g_i}^{*}=p_{K_0}.\]
By using Lemmas \ref{lem:ce1} and \ref{lem:local basis},
we get Theorem \ref{prop:GBS}.
\end{proof}

\begin{rem}
For the group $\overline{\pi}_{1}$ in Theorem \ref{prop:GBS} and the modular function $\Delta$ of $\overline{\pi}_{1}$,
$\ker\Delta$ is an open subgroup of $\overline{\pi}_{1}$.
When the graph $Y$ is countable,
$\ker\Delta$ satisfies the conditions in Proposition \ref{prop;elementary},
and $\overline{\pi}_{1}$ is an elementary group in  Wesolek's sense \cite{W}.
That is shown as follows.
Since $\pi_{1}$ is countable,
we can take a following decreasing sequence $(K_n)_{n=1}^\infty$ of compact open subgroups in $\overline{\pi}_{1}$ by using the proof of Theorem \ref{prop:GBS}.
\begin{itemize}
\item $\bigcap_{n=1}^{\infty}K_{n}=\{e\}$.
\item For every $n$, there is a finite subset $F_{n}$ of $\pi_{1}$ with $K_{n}=\bigcap_{h\in F_{n}}\textbf{Stab}(hG_{t(y)})\cap\overline{G}_{t(y)}$.
\item For every $n$, self-adjoint element $x\in\rg(\overline{\pi}_{1})$ and $\epsilon>0$,
there are $g_{1}, g_{2},\dots,g_{m}\in\pi_{1}\cap\ker\Delta$ such that each $g_{i}$ commutes with all elements of $K_{n}$ and 
\[\norm{\frac{1}{m}\sum_{i=1}^{m}\lambda_{g_i}(x-E_{K_{n}}(x))\lambda_{g_i}^*}<\epsilon\] holds.
\end{itemize}
Define the open subset $L_{n}:=\ker\Delta\cap\{g\in\overline{\pi}_{1}\mid g^{-1}K_{n}g\leq\overline{G}_{t(y)}\}$ for every $n$,
then we can show that $L_{n}$ is the normalizer $N_{\overline{\pi}_{1}}(K_{n})$ of $K_{n}$ as follows.
Since $L_{n}$ and $N_{\overline{\pi}_{1}}(K_{n})$ are clopen subsets of $\overline{\pi}_{1}$,
it suffices to show that $L_{n}\cap\pi_{1}\subset{N_{\overline{\pi}_{1}}(K_{n})}$.
Put $K_{n}^{\prime}:=\pi_{1}\cap K_{n}$,
and take $g\in L_{n}\cap\pi_{1}$,
then we have $g^{-1}K_{n}^{\prime}g=G_{t(y)}\cap g^{-1}K_{n}g$.
Since $g\in\ker\Delta$,
we get 
\[
[G_{t(y)}:g^{-1}K_{n}^{\prime}g]=[\overline{G}_{t(y)}:g^{-1}K_{n}g]
=[\overline{G}_{t(y)}:K_{n}]=[G_{t(y)}:K_{n}^{\prime}].
\]
Thus we get $K_{n}^{\prime}=g^{-1}K_{n}^{\prime}g\leq G_{t(y)}\cong\IZ$ and $g\in N_{\overline{\pi}_{1}}(K_{n})$.
Since $\bigcap_{n=1}^{\infty}K_{n}=\{e\}$ holds,
every $g\in\ker\Delta\cap\pi_{1}$ has $n\in\IN$ with $K_{n}\leq\textbf{Stab}(gG_{t(y)})$.
This gives $\ker\Delta=\bigcup_{n=1}^{\infty}L_{n}$.
By the third condition of $(K_n)_{n=1}^\infty$,
the \Cs-algebra $\rg(L_{n}/K_{n})\cong\rg(L_{n})p_{K_{n}}\subset\rg(\overline{\pi}_{1})$ is simple for any $n$.
Therefore,
the increasing sequence $(L_n)_{n=1}^\infty$ and the decreasing sequence $(K_n)_{n=1}^\infty$ satisfy the assumptions of Proposition \ref{prop;elementary}.
Since $\overline{\pi}_{1}/\ker\Delta\cong\Delta(\overline{\pi}_{1})$ is a countable discrete group,
the locally compact group $\overline{\pi}_{1}$ is elementary in  Wesolek's sense \cite{W}.
\end{rem}

\subsection*{Acknowledgements}
The author is grateful to her supervisor  Professor Yasuyuki Kawahigashi 
for his useful comments and constant support.
She is also grateful to Professor Narutaka Ozawa for helpful comments, Professor Yuhei Suzuki for invaluable help and Yuta Michimoto for helpful discussion.

\end{document} 
\end{document}